\documentclass[a4paper,10pt]{amsart}
\usepackage[english]{babel}
\usepackage[latin1]{inputenc}
\usepackage{amsmath,amsfonts,amssymb,amsthm,amscd,array,stmaryrd,mathrsfs,bbm}
\usepackage[all]{xy}
\usepackage{anysize}\marginsize{25mm}{25mm}{30mm}{30mm}
\allowdisplaybreaks[4]
\usepackage{textcomp,booktabs}
\usepackage{color}
\usepackage{tikz}
\usepackage{mathdots}
\usepackage{pst-node} 
\usepackage{geometry}
\usepackage{extarrows}
\usepackage{amsthm,amsmath,amssymb}
\usepackage{mathrsfs}
\usepackage{cite}
\usepackage{tikz}
\usepackage{enumitem}

\newtheorem{theorem}{Theorem}[section]
\newtheorem{corollary}[theorem]{Corollary}
\newtheorem{lemma}[theorem]{Lemma}

\newtheorem{example}{Example}[section]

\title{Additive Biderivations of Incidence Algebras}
\date{\today}

\author{Zhipeng Guan$^{1}$, Chi Zhang*$^{1}$}\thanks{*Corresponding author}

\address{$^{2}$College of Sciences, Northeastern University, Shenyang, Liaoning 110004, China}
\email{2300118@stu.neu.edu.cn, zhangchi@mail.neu.edu.cn}

\begin{document}
	
	\maketitle

	\begin{abstract}
		Let $\mathcal{R}$ be a commutative ring with unity, and let $P$ be a locally finite poset. The aim of the paper is to provide an explicit description of the additive biderivations of the incidence algebra $I(P, \mathcal{R})$. We demonstrate that every additive biderivation is the sum of several inner biderivations and extremal biderivations. Furthermore, if the number of elements in any maximal chain in $P$ is infinite, every additive biderivation of $I(P,\mathcal{R})$ is the sum of several inner biderivations.
	\end{abstract}
	
\tableofcontents

	\section{Introduction}
	
	The concept of incidence algebras of partially ordered sets was first introduced by Rota in the 1960s \cite{cite1} as a tool for solving combinatorial problems. Incidence algebras are fascinating objects that have been the subject of extensive research since their inception. For instance, Pierre and John \cite{cite3} describe the relationships between the algebraic properties of incidence algebras and the combinatorial features of the partially ordered sets. In \cite{cite2}, Spiegel \emph{et al.} provide a detailed analysis of the maximal and prime ideals, derivations and isomorphisms, radicals, and additional ring-theoretic properties of incidence algebras. Further developments on the structure of incidence algebras can be found in \cite{cite4,cite5,cite6,cite7}.
	
	The study of derivations in the context of algebras is a valuable and significant endeavor. Yang provides a detailed account of the structure of nonlinear derivations on incidence algebras in \cite{cite12}, specifically decomposing the nonlinear derivations into three more specific forms. Regarding the decomposition of derivations, two significant findings pertaining to their structure have been established. Notably, Baclawski demonstrated that every derivation of the incidence algebra $I(P, \mathcal{R})$, when $\mathcal{R}$ is a field and $P$ is a finite locally poset, can be expressed as the sum of an inner derivation and an additive induced derivation \cite{cite8}. This result was extended by Spiegel and O'Donnell \cite{cite2} to cases where $\mathcal{R}$ is a commutative ring. Additional insights into the structure of other special derivations on incidence algebra can be found in \cite{cite11,cite13,cite14,cite15}.
	
	Building upon the aforementioned studies of derivations on incidence algebras, it is natural to investigate biderivations in this context. Kaygorodov \emph{et al.} delineate the structure of antisymmetric biderivations of finite incidence algebras, which is equivalent to incidence algebras when $P$ is finite, in \cite{Poisson}. In \cite{cite9}, Benkovi\v{c} proves that every biderivation of a triangular algebra is the sum of an inner biderivation and an external biderivation. Later, Ghosseiri demonstrates that every biderivation of upper triangular matrix rings is the sum of an inner biderivation, an external biderivation, and a distinct category of biderivations \cite{cite10}. Ghosseiri also presents particular instances where every biderivation is inner.
	
	In light of the preceding discussion on the decomposition of biderivations in the context of incidence algebras, our objective is to decompose the biderivations of incidence algebras into more specific forms.
	Specifically, let $\mathcal{R}$ be a commutative ring with unity, and let $P$ be a locally finite poset such that any maximal chain in $P$ contains at least three elements. 
	We decompose $P$ using the connected relation defined on $P$ and the chain to obtain $P = \bigcup_{i \in \mathcal{I}}\bigcup_{j \in \mathcal{J}_i}  P^i_j$.  
	For any biderivation $b$ of the incidence algebra $I(P, \mathcal{R})$ and for all $\alpha, \beta \in I(P, \mathcal{R})$, the biderivation $b$ satisfies
	\begin{equation}
		b(\alpha, \beta) = \sum_{i \in \mathcal{I}} \left(\sum_{j \in \mathcal{J}_i} \lambda^i_j [\alpha^i_j, \beta^i_j] + [\hat{\alpha}^i,[\hat{\beta}^i, T^i]]\right), \nonumber
	\end{equation} 
	where $\lambda^i_j \in \mathcal{R}$, $\alpha^i_j, \beta^i_j \in I(P^i_j, \mathcal{R})$, and $\hat{\alpha}^i,\hat{\beta}^i, T^i \in I(P^i,\mathcal{R})$.
	In other words, the additive biderivation $b$ is the sum of several inner biderivations and extremal biderivations exactly.

The following sections provide a comprehensive overview of the contents of this article. In Section 2, we introduce preliminary concepts related to additive biderivations of incidence algebras and present decompositions of connected and disconnected posets $P$. Section 3 presents lemmas concerning derivations of incidence algebras that will be employed in subsequent proofs. Section 4 analyzes the structure of additive biderivations of incidence algebras.
	
	\section{Preliminaries}

	\subsection{Incidence Algebra}
	
	Throughout this paper, $\mathcal{R}$ denotes a commutative ring with unity. Recall that a relation $\leq$ is said to be a \textit{partial order} on a set $S$ if it satisfies the following conditions:
	\begin{enumerate}
		\item Reflexivity: $a \leq a$ for all $a \in S$;
		\item Antisymmetry: if $a \leq b$ and $b \leq a$, then $a = b$;
		\item Transitivity: if $a \leq b$ and $b \leq c$, then $a \leq c$.
	\end{enumerate}
	A \textit{partially ordered set} (or \textit{poset}) is a set equipped with a partial order. A poset $(S, \leq)$ is \textit{locally finite} if, for any $x \leq y$ in $S$, there are finitely many elements $u \in S$ satisfying $x \leq u \leq y$. Let $(P, \leq)$ be a locally finite poset, and let $x < y$ denote that $x \leq y$ and $x \neq y$.
	
	The \textit{incidence algebra} $I(P, \mathcal{R})$ is defined as the set of functions:
	\begin{equation}
		I(P, \mathcal{R}) = \{ f : P \times P \rightarrow \mathcal{R} \mid f(x, y) = 0 \text{ if } x \nleq y \},
	\end{equation}
	with multiplication given by:
	\begin{equation}
		(fg)(x, y) = \sum_{x \leq z \leq y} f(x, z) g(z, y), \quad \text{for all } x, y \in P.
	\end{equation}
	
	For each pair $x \leq y$ in $P$, let $e_{xy}$ denote the element in $I(P, \mathcal{R})$ defined by:
	\begin{equation}
		e_{xy}(u, v) = 
		\begin{cases}
			1, & \text{if } (x, y) = (u, v); \\
			0, & \text{otherwise}.
		\end{cases}
	\end{equation}
	For brevity, we will use the notation $\alpha_{xy}$ to denote $\alpha(x, y)$, where $\alpha \in I(P, \mathcal{R})$ and $x, y \in P$. Consequently, any element $\alpha \in I(P, \mathcal{R})$ can be expressed as $\alpha = \sum_{x \leq y} \alpha_{xy} e_{xy}$. It is evident that the multiplication in $I(P, \mathcal{R})$ satisfies the following properties:
	\begin{enumerate}
		\item If $y = u$, then $e_{xy} e_{uv} = e_{xv}$;
		\item If $y \neq u$, then $e_{xy} e_{uv} = 0$.
	\end{enumerate}
	This relation allows us to derive the following formula, which will be used extensively in this article. For any $f \in I(P, \mathcal{R})$ and $x \leq y$, $u \leq v$ in $P$, we have
	\begin{equation}\label{eq:e_{xx}fe_{yy}=f_{xy}e_{xy}}
		e_{xy} f e_{uv} = f_{yv} e_{xv}.
	\end{equation}

	\subsection{Additive Biderivations}
	A function $d$ is a \textit{derivation} of $I(P, \mathcal{R})$ if, for any $\alpha, \beta \in I(P, \mathcal{R})$, it satisfies
	\begin{equation*}
		d(\alpha \beta) = \alpha d(\beta) + d(\alpha) \beta.
	\end{equation*}
	A function $b$ is said to be a \textit{biderivation} of $I(P, \mathcal{R})$ if it is a derivation when fixing any one of its arguments. This means that for every $\alpha, \beta, \gamma \in I(P, \mathcal{R})$, $b$ satisfies
	\begin{align*}
		&b(\alpha \beta, \gamma) = \alpha b(\beta, \gamma) + b(\alpha, \gamma) \beta, \\
		&b(\alpha, \beta \gamma) = \beta b(\alpha, \gamma) + b(\alpha, \beta) \gamma.
	\end{align*}
	Moreover, $b$ is an \textit{additive biderivation} if $b$ also satisfies
	\begin{align*}
		&b(\alpha + \beta, \gamma) = b(\alpha, \gamma) + b(\beta, \gamma), \\
		&b(\alpha, \beta + \gamma) = b(\alpha, \beta) + b(\alpha, \gamma).
	\end{align*}

	The so-called \textit{inner biderivation} is defined by:
	\begin{equation}
		b(\alpha, \beta) = \lambda [\alpha, \beta],
	\end{equation}
	where $\lambda \in \mathcal{R}$ is a fixed element and $[\alpha, \beta]$ denotes the Lie bracket of $\alpha$ and $\beta$.
	Beside, the biderivation $b$ is called \textit{extremal biderivation}, if there is $\gamma$ in $I(P,\mathcal{R})$, and 
	\begin{equation}
		b(\alpha, \beta) = [\alpha, [\beta, \gamma]].
	\end{equation}

	\subsection{Decomposition of Posets}
	
	Our research reveals a close relationship between the structure of additive biderivations in $I(P, \mathcal{R})$ and that of the poset $P$. Accordingly, we proceed to decompose $P$ in this section, which will subsequently be employed in constructing the structure of the additive biderivations. In detail, we will consider and decompose the cases where P is connected or not.

	First, let us introduce some notation. Consider a relation $\sim$ on $P$ where $x \sim y$ indicates that $x$ and $y$ are comparable, i.e., either $x \leq y$ or $y \leq x$. The relation $\nsim$ indicates that $x$ and $y$ are not comparable. A \textit{totally ordered set} is a poset in which every pair of elements is comparable. A \textit{chain} in a poset $P$ is a subset that is totally ordered with respect to $\leq$.
	
	\begin{example}\label{S}
		Consider the poset $S = \{a, b, c, d, e, f, g, h, i\}$ represented by the following Hasse diagram, where $x \rightarrow y$ denotes $x \leq y$ for all $x, y \in S$. In this case, the pair of elements $b$ and $e$ is not comparable, and the subset $\{g, h, i\}$ of $S$ constitutes a chain.
		
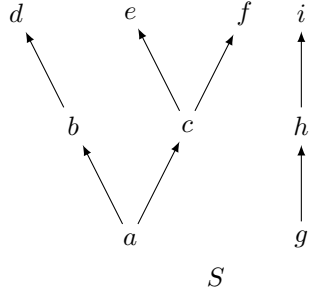
\begin{center}
			\begin{tikzpicture}
				\node(0) at (0,0) {$a$};
				\node(a) at (-0.75,1.5) {$b$};
				\node(c) at (0.75,1.5) {$c$};
				\node(ab) at (-1.5,3) {$d$};
				\node(cd) at (0,3) {$e$};
				\node(ce) at (1.5,3) {$f$};
				\node(1) at (2.25,0) {$g$};
				\node(2) at (2.25,1.5) {$h$};
				\node(3) at (2.25,3) {$i$};
				\node(remark) at (1.125,-0.5) {$S$};
				
				\draw [-latex] (0)--(a);
				\draw [-latex] (a)--(ab);
				\draw [-latex] (0)--(c);
				\draw [-latex] (c)--(cd);
				\draw [-latex] (c)--(ce);
				\draw [-latex] (1)--(2);
				\draw [-latex] (2)--(3);
			\end{tikzpicture}
		\end{center} 
	\end{example}

	\noindent \textbf{1. The first decomposition, when P is not connected.}

	In this context, two elements $x, y \in P$ are defined as being \textit{connected} if there exists a sequence $u_0, u_1, \dots, u_n \in P$ such that $x \sim u_0$, $u_0 \sim u_1$, $\dots$, $u_{n-1} \sim u_n$, and $u_n \sim y$. A poset $P$ is \textit{connected} if any pair of elements in it are connected.
	
	It is evident that the relation of being connected is an equivalence relation on $P$. Thus, we can decompose $P$ into the union of its connected components:
	\begin{equation}
		P = \bigcup_{i \in \mathcal{I}} P^i, \label{1_decomposition}
	\end{equation}
	where $\mathcal{I}$ is an index set and each $P^i$ is a connected poset.
	
	\begin{example}
		Consider the poset $S$ defined in Example \ref{S}. It can be observed that $S$ can be expressed as the union of two disjoint sets: $S^1 = \{a, b, c, d, e, f\}$ and $S^2 = \{g, h, i\}$.
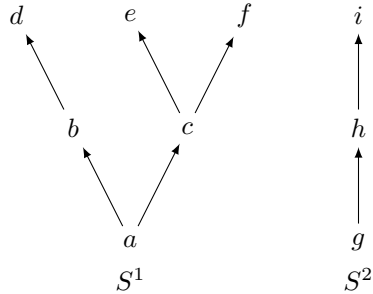
		\begin{center}
			\begin{tikzpicture}
				\node(0) at (0,0) {$a$};
				\node(a) at (-0.75,1.5) {$b$};
				\node(c) at (0.75,1.5) {$c$};
				\node(ab) at (-1.5,3) {$d$};
				\node(cd) at (0,3) {$e$};
				\node(ce) at (1.5,3) {$f$};
				\node(1) at (3,0) {$g$};
				\node(2) at (3,1.5) {$h$};
				\node(3) at (3,3) {$i$};
				\node(remark) at (0,-0.5) {$S^1$};
				\node(remark1) at (3,-0.5) {$S^2$};
				
				\draw [-latex] (0)--(a);
				\draw [-latex] (a)--(ab);
				\draw [-latex] (0)--(c);
				\draw [-latex] (c)--(cd);
				\draw [-latex] (c)--(ce);
				\draw [-latex] (1)--(2);
				\draw [-latex] (2)--(3);
			\end{tikzpicture}
		\end{center} 
	\end{example}
	
	It is evident that $P^i \cap P^j = \varnothing$ for $i \neq j$. Therefore, for any $\alpha \in I(P, \mathcal{R})$, we have
	\begin{equation}
		\alpha = \sum_{i \in \mathcal{I}} \alpha^i, \quad \alpha^i = \sum_{x \leq y \in P^i} \alpha_{xy} e_{xy} \in I(P^i, \mathcal{R}). \label{1_alpha_decomposition}   
	\end{equation}

	\noindent \textbf{2. The second decomposition, when P is connected.}
	
	Next, we present a second decomposition of $P$ when $P$ is connected. We begin by defining a key term. A chain in $P$ is called \textit{maximal} if adding any element from $P$ to it would result in it no longer being a chain.
	In this context, the notation $\{x, y, \dots\}^+$ will be used to denote a maximal chain in $P$ that contains $\{x, y, \dots\}$.
	
	We proceed to define a set of subsets of $P$:
	\begin{equation}
		L = \{ l \subset P \mid l \text{ is a maximal chain} \}.
	\end{equation}
	For any pair of elements $l', l'' \in L$, we define the relation $l' \approx l''$ if there exist elements $x < y$ in $P$ such that $x, y \in l' \cap l''$. We say $l'$ and $l''$ are \textit{connected} if there exist chains $l_0, l_1, \dots, l_n \in L$ such that $l' \approx l_0$, $l_0 \approx l_1$, $\dots$, $l_{n-1} \approx l_n$, $l_n \approx l''$.
	
	With respect to the connectedness relation on $L$, we can decompose $L$ into the union of its equivalence classes:
	\begin{equation}
		L = \bigcup_{j \in \mathcal{J}} L_j, \label{2_L_decomposition} 
	\end{equation}
	where $\mathcal{J}$ is an index set corresponding to the decomposition of $L$.
	
	For each $j \in \mathcal{J}$, define a subset of $P$:
	\begin{equation}
		P_j = \{ x \in P \mid \text{there exists } l \in L_j \text{ such that } x \in l \}. \label{2_decomposition}
	\end{equation}
	It is evident that $P = \bigcup_{j \in \mathcal{J}} P_j$. However, it is not necessarily the case that $P_i \cap P_j = \varnothing$ for all $i \neq j$ in $\mathcal{J}$.
	
	\begin{example}
		To illustrate this decomposition, consider $S^1 = \{ a, b, c, d, e, f \}$ from the previous example. The set $S^1$ can be expressed as the union of two sets: $S^1_1 = \{ a, b, d \}$ and $S^1_2 = \{ a, c, e, f \}$.
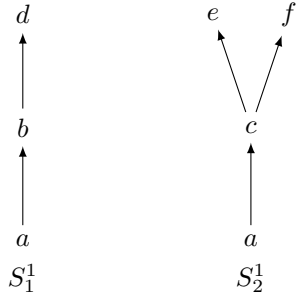
		\begin{center}
			\begin{tikzpicture}
				\node(0) at (-1.5,0) {$a$};
				\node(0') at (1.5,0) {$a$};
				\node(a) at (-1.5,1.5) {$b$};
				\node(c) at (1.5,1.5) {$c$};
				\node(ab) at (-1.5,3) {$d$};
				\node(cd) at (1,3) {$e$};
				\node(ce) at (2,3) {$f$};
				\node(remark) at (-1.5,-0.5) {$S^1_1$};
				\node(remark2) at (1.5,-0.5) {$S^1_2$};
				
				\draw [-latex] (0)--(a);
				\draw [-latex] (a)--(ab);
				\draw [-latex] (0')--(c);
				\draw [-latex] (c)--(cd);
				\draw [-latex] (c)--(ce);
			\end{tikzpicture}
		\end{center} 
	\end{example}
	
	An element $x \in P$ is said to be \textit{maximal} if there is no element $y \in P$ such that $x < y$, and \textit{minimal} if there is no element $y \in P$ such that $y < x$.
	
	\begin{lemma}\label{P_iP_j}
		If $i \neq j$ in $\mathcal{J}$ defined in \eqref{2_L_decomposition}, there does not exist a pair of elements $x < y$ in $P$ such that $x, y \in P_i \cap P_j$. Additionally, any element in $P_i \cap P_j$ is either a minimal or maximal element of $P$.
	\end{lemma}
	\begin{proof}
		We begin by asserting that any element in $P_i \cap P_j$ is isolated, meaning that no other element in $P_i \cap P_j$ can be compared with it. Suppose, for contradiction, that there exists a pair of elements $x < y$ contained in $P_i \cap P_j$. Then there exist chains $l_1 \in L_i$ and $l_2 \in L_j$ such that $x, y \in l_1$ and $x, y \in l_2$. It follows that $l_1 \approx l_2$, implying $P_i = P_j$ and $i = j$, which contradicts the assumption that $i \neq j$. Therefore, any element in $P_i \cap P_j$ is isolated.
		
		Now, suppose there exists an element $x \in P_i \cap P_j$ that is neither maximal nor minimal. Then there exist elements $u < x < v$ in $P$. According to the definition of $P_i$ and $P_j$ and the assumption that any maximal chain of $P$ has at least three elements, there exist elements $y_i \in P_i$ and $y_j \in P_j$ that can be compared with $x$, and there exist maximal chains $l_i \in L_i$ and $l_j \in L_j$ such that $x, y_i \in l_i$ and $x, y_j \in l_j$. Without loss of generality, suppose $x < y_i$. It is evident that $\{ u, x, v \}^+ \approx \{ u, x, y_i \}^+ \approx l_i$, so $u, x, v \in P_i$. Similarly, $u, x, v \in P_j$. This implies that both $u < x$ and $x < v$ are elements of $P_i \cap P_j$, which contradicts the previous conclusion. Thus, any element in $P_i \cap P_j$ must be either minimal or maximal.
	\end{proof}
	
	By the lemma above, for any $\alpha \in I(P, \mathcal{R})$, where $P$ is connected, we can decompose it as
	\begin{equation}
		\alpha = \sum_{j \in \mathcal{J}} \alpha'_j + \alpha^D, \quad \text{where } \alpha'_j = \sum_{x < y \in P_j} \alpha_{xy} e_{xy}, \text{ and } \alpha^D = \sum_{z \in P} \alpha_{zz} e_{zz}. \label{eq:2_alpha_decomposition}
	\end{equation}
	If $i \neq j$ and the product $\alpha'_i \beta'_j = \sum_{x < y \in P_i} \sum_{u < v \in P_j} \alpha_{xy} \beta_{uv} e_{xy} e_{uv}$ is non-zero, then there exist $x' < y'$ in $P_i$ and $u' < v'$ in $P_j$ such that $e_{xy} e_{uv} \neq 0$. Therefore, we conclude that $y = u$, implying that $y \in P_i \cap P_j$ and $x < y < v$, which contradicts Lemma \ref{P_iP_j}. This leads to the following corollary.
	
	\begin{corollary}
		For any $i \neq j$ in $\mathcal{J}$, let $\alpha'_i \in I(P_i, \mathcal{R})$ and $\beta'_j \in I(P_j, \mathcal{R})$, then $\alpha'_i \beta'_j = 0$.  \label{a_ib_j}
	\end{corollary}
	
	\subsection{Summary of Notations}
	
	For clarity, we summarize the symbols and definitions that will be used frequently in the following sections:
	\begin{itemize}
		\item For any $x, y \in P$, $x \sim y$ indicates that $x$ and $y$ are comparable.
		\item Two elements $x$ and $y$ are \textit{connected} if there exist elements $u_0, u_1, \dots, u_n \in P$ such that $x \sim u_0$, $u_0 \sim u_1$, $\dots$, $u_{n-1} \sim u_n$, and $u_n \sim y$. A poset $P$ is \textit{connected} if any pair of elements in it are connected.
		\item An element $x$ is \textit{maximal} if there exists no element $y$ such that $x < y$, and \textit{minimal} if there exists no element $y$ such that $y < x$.
		\item A \textit{chain} in $P$ is a totally ordered subset of $P$.
		\item A chain is \textit{maximal} if adding any element from $P$ to it would result in it no longer being a chain.
		\item For any pair of chains $l', l''$ in $P$, we say $l' \approx l''$ if there exist elements $x < y \in P$ such that $x, y \in l' \cap l''$. We say $l'$ and $l''$ are \textit{connected} if there exist chains $l_0, l_1, \dots, l_n \in P$ such that $l' \approx l_0$, $l_0 \approx l_1$, $\dots$, $l_{n-1} \approx l_n$, and $l_n \approx l''$.
	\end{itemize}
	
	\section{Derivations of $I(P, \mathcal{R})$}
	
	In the study of additive biderivations, it is necessary to discuss some properties of additive derivations of $I(P, \mathcal{R})$, since an additive biderivation becomes an additive derivation when one of its arguments is fixed. Let $P$ be a locally finite poset, let $\mathcal{R}$ be a commutative ring with unity, and let $I(P, \mathcal{R})$ be the incidence algebra of $P$ over $\mathcal{R}$. Suppose $d$ is an additive derivation of $I(P, \mathcal{R})$.
	
	We will denote $d_{xy}(\alpha)$ to represent $d(\alpha)(x, y)$ for any $\alpha \in I(P, \mathcal{R})$ and $x, y \in P$. We begin by demonstrating a lemma that shows the structure of the value of $d(r e_{xy})$.
	
	\begin{lemma}\label{d_structure}
		For any $x \leq y \in P$ and $r \in \mathcal{R}$, we have 
		\begin{equation}
			d(r e_{xy}) = \sum_{p \leq x} d_{py}(r e_{xy}) e_{py} + \sum_{y < q} d_{xq}(r e_{xy}) e_{xq}. \label{eq:10}
		\end{equation}
	\end{lemma}
	\begin{proof}
		Consider a fixed pair $x \leq y \in P$ and an arbitrary element $r \in \mathcal{R}$. Suppose $u \leq v \in P$ with $u \neq x$ and $v \neq y$. We examine the expression $e_{uu} d(r e_{xy}) e_{vv}$ by multiplying $d(r e_{xy})$ on the left by $e_{uu}$ and on the right by $e_{vv}$:
		
		\begin{equation}
			\begin{aligned}
				e_{uu} d(r e_{xy}) e_{vv} &= e_{uu} d(r e_{xx} e_{xy}) e_{vv} \\
				&= e_{uu} e_{xx} d(r e_{xy}) e_{vv} + e_{uu} d(e_{xx}) r e_{xy} e_{vv} \\
				&= 0.
			\end{aligned}
		\end{equation}
		
		Since $e_{uu} d(r e_{xy}) e_{vv} = d_{uv}(r e_{xy}) e_{uv}$, it follows that $d_{uv}(r e_{xy}) = 0$ unless $u = x$ or $v = y$. This establishes the desired equality \eqref{eq:10}.
	\end{proof}
	
	The proof of the aforementioned lemma allows us to derive the following corollary.
	
	\begin{corollary}\label{d_uv(e_xy)=0}
		For any $x \leq y$, $u \leq v \in P$, we have $d_{uv}(r e_{xy}) = 0$ unless the elements $x, y, u, v$ satisfy one of the following cases: (1) $u = x$, $v \leq y$; (2) $u \leq x$, $v = y$.
	\end{corollary}
	
	In describing the structure of biderivations, we will try to prove that some of them are equivalent. To do this, we introduce a number of instrumental lemmas as described below.
	
	\begin{lemma}\label{d_structure_like_linear}
		For any $x \leq y \in P$ and $r \in \mathcal{R}$, we have:
		\begin{enumerate}
			\item[(1)] For any $p < x$, $d_{py}(r e_{xy}) = rd_{py}(e_{xy}) = r d_{px}(e_{xx})$;
			\item[(2)] For any $q > y$, $d_{xq}(r e_{xy}) = rd_{xq}(e_{xy}) = r d_{yq}(e_{yy})$.
		\end{enumerate}
	\end{lemma}
	\begin{proof}
		We will demonstrate part (1), and the proof of part (2) follows analogously. Consider fixed elements $x \leq y \in P$ and $r \in \mathcal{R}$. For any $p < x$, using equation \eqref{eq:e_{xx}fe_{yy}=f_{xy}e_{xy}}, we obtain
		\begin{equation}
			\begin{aligned}
				d_{py}(r e_{xy}) e_{py} &= e_{pp} d(r e_{xy}) e_{yy} \\
				&= e_{pp} \left( e_{xx} d(r e_{xy}) + d(e_{xx}) r e_{xy} \right) e_{yy} \\ 
				&= r d_{px}(e_{xx}) e_{py}.
			\end{aligned}
		\end{equation}
		Therefore, we deduce that $d_{py}(r e_{xy}) = r d_{px}(e_{xx})$. Setting $r$ to be the unity element of $\mathcal{R}$ yields $d_{py}(e_{xy}) = d_{px}(e_{xx})$, thereby establishing part (1).
	\end{proof}
	
	\begin{lemma}\label{d_xy(e_xx)+d_xy(e_yy)=0}
		For any $x < y \in P$, we have $d_{xy}(e_{xx}) + d_{xy}(e_{yy}) = 0$.
	\end{lemma}
	\begin{proof}
		Let $x < y \in P$ be fixed. Observe that
		\begin{equation}
			\begin{aligned}
				0 &= e_{xx} d(e_{xx} e_{yy}) e_{yy} \\
				&= e_{xx} d(e_{xx}) e_{yy} + e_{xx} d(e_{yy}) e_{yy} \\
				&= \left( d_{xy}(e_{xx}) + d_{xy}(e_{yy}) \right) e_{xy}.
			\end{aligned}
		\end{equation}
		Since the above expression equals zero and $e_{xy}$ is a non-zero element of the incidence algebra, it follows that $d_{xy}(e_{xx}) + d_{xy}(e_{yy}) = 0$.
	\end{proof}
	
	\begin{lemma}\label{d_xz=d_xy+d_yz}
		For any $x < y < z \in P$, we have $d_{xz}(e_{xz}) = d_{xy}(e_{xy}) + d_{yz}(e_{yz})$.
	\end{lemma}
	\begin{proof}
		Let $x < y < z \in P$ be fixed. Then,
		\begin{equation}
			\begin{aligned}
				d_{xz}(e_{xz}) e_{xz} &= e_{xx} d(e_{xy} e_{yz}) e_{zz} \\
				&= e_{xx} \left( d(e_{xy}) e_{yz} + e_{xy} d(e_{yz}) \right) e_{zz} \\
				&= \left( d_{xy}(e_{xy}) + d_{yz}(e_{yz}) \right) e_{xz}.
			\end{aligned}
		\end{equation}
		Since $e_{xz}$ is a non-zero element of the incidence algebra, we deduce that $d_{xz}(e_{xz}) = d_{xy}(e_{xy}) + d_{yz}(e_{yz})$.
	\end{proof}
	
	\section{Biderivations of $I(P, \mathcal{R})$}
	
	In this section, we will employ the properties of derivations derived in Section 3 to prove our main theorem (Theorem \ref{final}), which elucidates the structure of additive biderivations on the incidence algebra $I(P, \mathcal{R})$. Let $P$ be a locally finite poset, let $\mathcal{R}$ be a commutative ring with unity. In this section, we will use the notation $b_{xy}(\alpha, \beta)$ to denote the value $b(\alpha, \beta)(x, y)$ for any $\alpha, \beta \in I(P, \mathcal{R})$ and $x, y \in P$.
	
	We begin by considering a corollary that can be readily extended from Lemma \ref{d_structure}.
	
	\begin{corollary}\label{corollary1}
		For any $x \leq y$, $u \leq v \in P$, and $r_1, r_2 \in \mathcal{R}$, we have
		\begin{equation}
			b(r_1 e_{xy}, r_2 e_{uv}) = 
			\begin{cases}
				b_{xv}(r_1 e_{xy}, r_2 e_{uv}) e_{xv} + b_{uy}(r_1 e_{xy}, r_2 e_{uv}) e_{uy}, & \text{if } x \neq u, y \neq v; \\
				\displaystyle \sum_{y \leq q, v \leq q} b_{xq}(r_1 e_{xy}, r_2 e_{xv}) e_{xq}, & \text{if } x = u, y \neq v; \\
				\displaystyle \sum_{p \leq x, p \leq u} b_{py}(r_1 e_{xy}, r_2 e_{uy}) e_{py}, & \text{if } x \neq u, y = v; \\
				\displaystyle \sum_{p \leq x} b_{py}(r_1 e_{xy}, r_2 e_{xy}) e_{py} + \sum_{y < q} b_{xq}(r_1 e_{xy}, r_2 e_{xy}) e_{xq}, & \text{if } x = u, y = v.
			\end{cases} \label{eq:corollary1}
		\end{equation} 
	\end{corollary}
	
	The preliminary stage will entail a reduction in the complexity of the structure of $b$. This will be achieved by establishing a sufficient condition that characterizes the circumstances under which b is equal to zero.

	\begin{lemma}\label{cant compare}
		For $x \leq y$, $u \leq v \in P$, if at least one pair among $\{x, y, u, v\}$ is not comparable, then $b(r_1 e_{xy}, r_2 e_{uv}) = 0$ for any $r_1, r_2 \in \mathcal{R}$. 
	\end{lemma}
	\begin{proof}
		Consider $x \leq y$, $u \leq v \in P$. Suppose at least one pair among $\{x, y, u, v\}$ is not comparable. This situation can be divided into four cases: (1) $x \nsim u$; (2) $x \nsim v$; (3) $y \nsim u$; (4) $y \nsim v$. 
		
		\begin{itemize}[align=left, itemindent=-2em]
			\item[\textbf{Case 1}:] Suppose $x \nsim u$. We consider two subcases: $y \neq v$ and $y = v$.
			\par If $y \neq v$, by Corollary \ref{corollary1}, we have
			\begin{equation}
				b(r_1 e_{xy}, r_2 e_{uv}) = b_{xv}(r_1 e_{xy}, r_2 e_{uv}) e_{xv} + b_{uy}(r_1 e_{xy}, r_2 e_{uv}) e_{uy}.
			\end{equation}
			According to Corollary \ref{d_uv(e_xy)=0}, both $b_{xv}(r_1 e_{xy}, r_2 e_{uv})$ and $b_{uy}(r_1 e_{xy}, r_2 e_{uv})$ are zero because $x \nsim u$. 
			
			\par If $y = v$, we obtain that
			\begin{equation}
				\begin{aligned}
                    &~~~~b(r_1 e_{xy}, r_2 e_{uy}) \\
					&= \sum_{p \leq x, p \leq u} b_{py}(r_1 e_{xy}, r_2 e_{uy}) e_{py} \\
					&= \sum_{p < x, p < u} b_{py}(r_1 e_{xy}, r_2 e_{uy}) e_{py} + b_{xy}(r_1 e_{xy}, r_2 e_{uy}) e_{xy} + b_{uy}(r_1 e_{xy}, r_2 e_{uy}) e_{uy}.
				\end{aligned} \label{eq:20}
			\end{equation}
			In \eqref{eq:20}, both $b_{xy}(r_1 e_{xy}, r_2 e_{uy}) e_{xy}$ and $b_{uy}(r_1 e_{xy}, r_2 e_{uy}) e_{uy}$ are zero by Corollary \ref{d_uv(e_xy)=0}. For the remaining terms in \eqref{eq:20}, using Lemma \ref{d_structure_like_linear} on the first argument of each term, we have
			\begin{equation}
				\sum_{p < x, p < u} b_{py}(r_1 e_{xy}, r_2 e_{uy}) e_{py} = \sum_{p < x, p < u} b_{px}(r_1 e_{xx}, r_2 e_{uy}) e_{py} = 0. \label{eq:21}
			\end{equation} 
			Noting that $p < u$ and $x < y$ (If $x = y$, we get $u < x = v$, that contradicts $x \nsim u$.), $b_{px}(r_1 e_{xx}, r_2 e_{uy}) e_{py}$ in \eqref{eq:21} is zero by Corollary \ref{d_uv(e_xy)=0}. Therefore, the lemma is proved when $x \nsim u$.
			
			\item[\textbf{Case 2}:] If $x \nsim v$, then $x \neq u$ and $y \neq v$. From Corollary \ref{corollary1}, we have
			\begin{equation}
				b(r_1 e_{xy}, r_2 e_{uv}) = b_{uy}(r_1 e_{xy}, r_2 e_{uv}) e_{uy}. 
			\end{equation}
			We assert that either $u < x$ or $u \nsim x$, because if $x \leq u$, then $x \leq u \leq v$, which contradicts $x \nsim v$. For $b_{uy}(r_1 e_{xy}, r_2 e_{uv})$, it is zero by Lemma \ref{d_structure} if $u \nsim x$. If $u < x$, by Lemma \ref{d_structure}, we get
			\begin{equation}
				b_{uy}(r_1 e_{xy}, r_2 e_{uv}) = b_{ux}(r_1 e_{xx}, r_2 e_{uv}) = 0.
			\end{equation}
			Therefore, the lemma is proved when $x \nsim v$.
			
			\item[\textbf{Cases 3 and 4}:] If $y \nsim u$ or $y \nsim v$, the proof is similar to Cases 1 and 2.
		\end{itemize} 
		Thus, the lemma is proved.
	\end{proof}

	According to the decomposition \eqref{1_alpha_decomposition}, we can suppose that for any $\alpha, \beta \in I(P, \mathcal{R})$,
	\begin{align*}
		\alpha = \sum_{i \in \mathcal{I}} \alpha^i, \qquad
		\beta = \sum_{i \in \mathcal{I}} \beta^i,
	\end{align*}
	where $\alpha^i, \beta^i \in I(P^i, \mathcal{R})$. 
	By applying the decomposition of $P$, as defined in \eqref{1_decomposition}, and the aforementioned lemma, it is straightforward to conclude that $b(\alpha^{i_1}, \beta^{i_2}) = 0$ if $i_1 \neq i_2 \in \mathcal{I}$, thereby yielding the following result:
	\begin{equation}
		b(\alpha, \beta) = \sum_{i \in \mathcal{I}} b(\alpha^i, \beta^i). \label{eq:decomposition1} 
	\end{equation}

	Accordingly, this decomposition permits us to limit our analysis to the case where P is connected, as will be discussed subsequently.
	We will subsequently present the theorem, in the case where P is connected, which describes the conditions under which certain terms in equation \eqref{eq:corollary1} are equal to zero. Prior to this, we will present an instrumental lemma.

	\begin{lemma}\label{change_seat}
		For any $\alpha, \beta, \gamma, \delta \in I(P, \mathcal{R})$, we have 
		\begin{equation}
			b(\alpha, \beta) [\gamma, \delta] = [\alpha, \beta] b(\gamma, \delta). \label{eq:change_seat}
		\end{equation}
	\end{lemma}
	\begin{proof}
		Let $\alpha, \beta, \gamma, \delta \in I(P, \mathcal{R})$ be arbitrary. Since $b$ is a biderivation, it satisfies the derivation property in each argument separately. We begin by applying the derivation property to the first argument of $b(\alpha \gamma, \beta \delta)$, followed by applying it to the second argument. This yields:
		\begin{equation}
			\begin{aligned}
				b(\alpha \gamma, \beta \delta) &= \alpha \cdot b(\gamma, \beta \delta) + b(\alpha, \beta \delta) \cdot \gamma \\
				&= \alpha \beta \cdot b(\gamma, \delta) + \alpha \cdot b(\gamma, \beta) \delta + \beta \cdot b(\alpha, \delta) \gamma + b(\alpha, \beta) \delta \gamma.
			\end{aligned} \label{eq:[]b_1}
		\end{equation}
		
		Conversely, we first apply the derivation property to the second argument and then to the first argument, obtaining:
		\begin{equation}
			\begin{aligned}
				b(\alpha \gamma, \beta \delta) &= \beta \cdot b(\alpha \gamma, \delta) + b(\alpha \gamma, \beta) \cdot \delta \\
				&= \beta \alpha \cdot b(\gamma, \delta) + \beta \cdot b(\alpha, \delta) \gamma + \alpha \cdot b(\gamma, \beta) \delta + b(\alpha, \beta) \gamma \delta. 
			\end{aligned} \label{eq:[]b_2}
		\end{equation}
		
		Subtracting equation \eqref{eq:[]b_1} from equation \eqref{eq:[]b_2} gives:
		\begin{equation}
			0 = [\alpha, \beta] \cdot b(\gamma, \delta) - b(\alpha, \beta) \cdot [\gamma, \delta],
		\end{equation}
		which rearranges to the desired identity:
		\begin{equation}
			b(\alpha, \beta) [\gamma, \delta] = [\alpha, \beta] b(\gamma, \delta).
		\end{equation}
	The proof is completed.
	\end{proof}

	In accordance with the aforementioned lemma, it is possible to select specific values for $x,y,u,v$ in order to demonstrate that $b(e_{xy},e_{uv}) = 0$.
	For $x, y, u, v \in P$ that are comparable with each other and have $[e_{xy}, e_{uv}]=0$, the following two subcases can be identified:(1) $x=y=u=b$;(2)$x\neq v, y \neq u$. 
	The following two theorems will address these subcases in greater detail.

	\begin{theorem}\label{b(e_{xx},e_{xx})}
		For any $x\in P$, we have
		\begin{equation}
			b(e_{xx}, e_{xx}) = 
			\begin{cases}
				\sum_{\substack{x < y,\\ y\text{ is max}}}b_{xy}(e_{xx}, e_{xx})e_{xx}, &$x$ \text{ is a minimal element;} \\
				\sum_{\substack{y < x,\\ y\text{ is min}}}b_{xy}(e_{xx}, e_{xx})e_{xx}, &$x$ \text{ is a maximal element;} \\
				0, &\text{otherwise}.
			\end{cases}
		\end{equation}
	\end{theorem}
	\begin{proof} We prove the lemma by considering three cases based on the position of $x$ in $P$: minimal, maximal, and neither.
		
		Case 1: $x$ is a minimal element in $P$.
		By Corollary \ref{corollary1}, we have: 
		\begin{equation} 
			b(e_{xx}, e_{xx}) = \sum_{x \leq q}b_{xq}(e_{xx}, e_{xx})e_{xq}. 
		\end{equation} 
		For any pair $x \leq y$ where $y$ is not a maximal element, let $z > y$. 
		By Lemma \ref{change_seat}, we know: 
		\begin{equation} 
					b_{xy}(e_{xx}, e_{xx})e_{xz} = b(e_{xx}, e_{xx})[e_{yz}, e_{zz}] = [e_{xx}, e_{xx}]b(e_{yz}, e_{zz}) = 0. 
		\end{equation} 
		Thus, when $x$ is a minimal element, $b(e_{xx}, e_{xx})$ reduces to the given form $\sum_{x < y, y\text{ is max}}b_{xy}(e_{xx}, e_{xx})e_{xx}$.
		
		Case 2: $x$ is a maximal element in $P$. The procedure is similar to Case 1.
		
		Case 3: $x$ is neither a minimal nor maximal element.
		In this case, there exist $y, z$ such that $y < x < z$. For any $p \leq x$, we have: \begin{equation} 
			e_{pp}b(e_{xx}, e_{xx})e_{xz} = e_{pp}b(e_{xx}, e_{xx})[e_{xz}, e_{zz}] = e_{pp}[e_{xx}, e_{xx}]b(e_{xz}, e_{zz}) = 0.
		\end{equation} 
		This implies $b_{px}(e_{xx}, e_{xx}) = 0$. 
		Similarly, for any $x \leq q$, we can deduce $b_{xq}(e_{xx}, e_{xx}) = 0$. 
		Hence, when $x$ is neither minimal nor maximal element, we have $ b(e_{xx}, e_{xx}) = 0$.
	
		Combining the three cases, the lemma is proven. 
	\end{proof}

	\begin{theorem}\label{[e_xy,e_uv]=0}
		Let $P$ be a connected poset such that any maximal chain has at least three elements. For any $r_1, r_2 \in \mathcal{R}$ and $x, y, u, v \in P$ that are comparable with each other and have $x \neq v, y \neq u$.
		The additive biderivation $b$ satisfies $b(r_1 e_{xy}, r_2 e_{uv}) = 0$, except in the case where $x = y \neq u = v$ and one of $x$ and $u$ is the maximal element in $P$ and the other is the minimal element.
	\end{theorem}
	\begin{proof}
		Let $r_1, r_2 \in \mathcal{R}$ and $x, y, u, v \in P$ satisfy the conditions of the theorem.
		We further divide this case into four subcases:
	
		\begin{enumerate}
			\item[\textbf{Case 1}] If $x \neq u$ and $y \neq v$, except in the case where $x = y \neq u = v$ and one of $x$ and $u$ is the maximal element in $P$ and the other is the minimal element. From Corollary \ref{corollary1}, we have
			\begin{equation}
				b(r_1 e_{xy}, r_2 e_{uv}) = b_{xv}(r_1 e_{xy}, r_2 e_{uv}) e_{xv} + b_{uy}(r_1 e_{xy}, r_2 e_{uv}) e_{uy}.
			\end{equation}
			We have $b_{xv}(r_1e_{xy},r_2e_{uv})=b_{yv}(r_1e_{yy},r_2e_{uv})$ by using Lemma \ref{d_structure_like_linear} in its first argument when $x\leq y<v$.
			If $v<y$, it is also correct because $b_{xv}(r_1e_{xy},r_2e_{uv})=b_{yv}(r_1e_{yy},r_2e_{uv})=0$ by corollary \ref{corollary1}.
			By the same reasoning in its second argument, we obtain
			$b_{yv}(r_1e_{yy},r_2e_{uv}) = b_{yu}(r_1e_{yy},r_2e_{uu})$. 
			Without loss of generality, we can assume that $y < u$.
			It is evidence that $y$ is not minimal or $u$ is not maximal, otherwise the case is excluded.
			Thus by Lemma \ref{d_xy(e_xx)+d_xy(e_yy)=0} and Theorem \ref{b(e_{xx},e_{xx})}, we have
			\begin{equation}
				b_{yu}(r_1e_{yy},r_2e_{uu}) = -r_1r_2b_{yu}(e_{yy},e_{yy}) = 0			\end{equation}
			Similarly, $b_{uy}(r_1e_{xy},r_2e_{uv})$ is also equal to 0.
			Therefore, we have $b(r_1e_{xy},r_2e_{uv})=0$ in this case. 
			
			\item[\textbf{Case 2}] If $x = u$ and $y \neq v$. We consider two subcases: $v < y$ and $y < v$.
			\begin{itemize}
				\item[(a)] If $v < y$, we have
				\begin{equation}
					b(r_1 e_{xy}, r_2 e_{xv}) = \sum_{v < y \leq q} b_{xq}(r_1 e_{xy}, r_2 e_{xv}) e_{xq} \overset{\text{Lemma }\ref{d_structure_like_linear}}{=} \sum_{v < y \leq q} b_{vq}(r_1 e_{xy}, r_2 e_{vv}) e_{xq}.
				\end{equation}
				Since $x = u < v$ and $q \geq y$, $b_{vq}(r_1 e_{xy}, r_2 e_{vv}) = 0$ by Lemma \ref{d_structure} and Corollary \ref{d_uv(e_xy)=0}. Therefore, $b(r_1 e_{xy}, r_2 e_{xv}) = 0$.
				
				\item[(b)] If $y < v$, the proof is similar.
			\end{itemize}
			
			\item[\textbf{Case 3}] If $x \neq u$ and $y = v$. The proof is similar to subcase 2.
			
			\item[\textbf{Case 4}] If $x = u$ and $y = v$, implying $x = u < y = v$. By Corollary \ref{corollary1}, we have
			\begin{equation}
				b(r_1 e_{xy}, r_2 e_{xy}) = \sum_{p < x} b_{py}(r_1 e_{xy}, r_2 e_{xy}) e_{py} + \sum_{y < q} b_{xq}(r_1 e_{xy}, r_2 e_{xy}) e_{xq} + b_{xy}(r_1 e_{xy}, r_2 e_{xy}) e_{xy}. \label{b_xy(e_xy,e_xy)}
			\end{equation}
			For any pair $p < x$, the term $b_{py}(r_1e_{xy},r_2e_{xy})e_{py}$ is equal to $b_{px}(r_1e_{xx},r_2e_{xy})e_{py}$ by Lemma \ref{d_structure_like_linear}. 
			Let us consider the second argument of it, because $p\neq x$ and $x\neq y$, its is equal to 0 by Corollary \ref{d_uv(e_xy)=0}. By the same way, for any $y < q$, we have $b_{xq}(r_1e_{xy},r_2e_{xy})e_{xq}=0$. 
			Thus,
			\begin{equation}
				b(r_1 e_{xy}, r_2 e_{xy}) = b_{xy}(r_1 e_{xy}, r_2 e_{xy}) e_{xy}.
			\end{equation}
			By the assumption on $P$, there exists an element $z \neq x, y$ such that $z$ is comparable with $x$ and $y$. We consider three cases:
			\begin{enumerate}
				\item If $z < x < y$, applying Lemma \ref{change_seat}, we obtain
				\begin{equation}
					b(e_{zx}, e_{xx}) [r_1 e_{xy}, r_2 e_{xy}] = [e_{zx}, e_{xx}] b(r_1 e_{xy}, r_2 e_{xy}).
				\end{equation}
				The left-hand side is zero, and the right-hand side equals $b_{xy}(r_1 e_{xy}, r_2 e_{xy}) e_{zy}$, implying $b(r_1 e_{xy}, r_2 e_{xy}) = 0$.
				
				\item If $x < y < z$, the proof is similar to (a).
				
				\item If $x < z < y$, applying Lemma \ref{d_xz=d_xy+d_yz}, we have
				\begin{equation}
					b_{xy}(r_1 e_{xy}, r_2 e_{xy}) = b_{xz}(r_1 e_{xz}, r_2 e_{xy}) + b_{zy}(r_1 e_{zy}, r_2 e_{xy}) \overset{\text{Lemma }\ref{d_structure}}{=} 0,
				\end{equation}
				Consider the second argument of $b_{xz}(r_1e_{xz},r_2e_{xy})$ and $b_{zy}(r_1e_{zy},r_2e_{xy})$ respectively Because $x<z<y$, we have $b_{xz}(r_1e_{xz},r_2e_{xy}) = b_{zy}(r_1e_{zy},r_2e_{xy}) = 0$ by Corollary \ref{d_uv(e_xy)=0}.
				This implies $b(r_1e_{xy},r_2e_{xy})=0$.
			\end{enumerate}

		\end{enumerate}
		Considering all the above cases, the theorem is proved.
	\end{proof}

	The objective of the forthcoming project is to provide evidence that certain components of equation \eqref{eq:corollary1} are equal. In particular, it will be demonstrated that $b_{xy}(e_{xy}, e_{xx}) = b_{uv}(e_{uv}, e_{uu})$ when $x, y, u, v \in P$ satisfy specified conditions.
	Prior to this, two lemmas will be proven.

	\begin{lemma}\label{equal_1}
		For any $x < y < z \in P$, we have
		\begin{equation}
			b_{xy}(e_{xy}, e_{yy}) = b_{yz}(e_{yy}, e_{yz}) \quad \text{and} \quad b_{xy}(e_{yy}, e_{xy}) = b_{yz}(e_{yz}, e_{yy}).
		\end{equation}
	\end{lemma}
	\begin{proof}
		For any $x < y < z \in P$, applying Lemma \ref{d_structure_like_linear} to the left or right argument of $b_{xz}(e_{xy}, e_{yz})$ separately, we have
		\begin{equation}
			b_{yz}(e_{yy}, e_{yz}) = b_{xz}(e_{xy}, e_{yz}) = b_{xy}(e_{xy}, e_{yy})
		\end{equation}	                       
		Thus, we have $b_{xy}(e_{xy}, e_{yy}) = b_{yz}(e_{yy}, e_{yz})$. Using a similar process for $b_{xz}(e_{yz}, e_{xy})$, we obtain the other equality.
	\end{proof}
	
	\begin{lemma}\label{equal_2}
		For any $x < y < z \in P$, we have
		\begin{enumerate}
			\item[(1)] $b_{xy}(e_{xy}, e_{xx}) = b_{yz}(e_{yz}, e_{yy}) = b_{xz}(e_{xz}, e_{xx})$; 
			\item[(2)] $b_{xy}(e_{xx}, e_{xy}) = b_{yz}(e_{yy}, e_{yz}) = b_{xz}(e_{xx}, e_{xz})$. 
		\end{enumerate} 
	\end{lemma}
	\begin{proof}
		Let $x < y < z \in P$. For $b_{xz}(e_{xz}, e_{xx})$ and $b_{xz}(e_{xz}, e_{zz})$, using Lemma \ref{d_xz=d_xy+d_yz} for its first argument, we obtain
		\begin{equation}
			\begin{cases}
				b_{xz}(e_{xz}, e_{xx}) = b_{xy}(e_{xy}, e_{xx}) + b_{yz}(e_{yz}, e_{xx}) = b_{xy}(e_{xy}, e_{xx}); \\
				b_{xz}(e_{xz}, e_{zz}) = b_{xy}(e_{xy}, e_{zz}) + b_{yz}(e_{yz}, e_{zz}) = b_{yz}(e_{yz}, e_{zz})
			\end{cases} \label{eq:equal_2_3}
		\end{equation}
		noting that $b_{yz}(e_{yz}, e_{xx})=b_{xy}(e_{xy}, e_{zz})=0$ by corollary \ref{corollary1}.		Additionally, we have $b_{xz}(e_{xz}, e_{xx}) = -b_{xz}(e_{xz}, e_{zz})$ by Lemma \ref{d_xy(e_xx)+d_xy(e_yy)=0}. Plugging the results from \eqref{eq:equal_2_3} into this, we find that
		\begin{equation}
			b_{xy}(e_{xy}, e_{xx}) = -b_{yz}(e_{yz}, e_{zz}),
		\end{equation}
		which implies $b_{xy}(e_{xy}, e_{xx}) = b_{yz}(e_{yz}, e_{yy})$. Then $b_{xy}(e_{xy}, e_{xx}) = b_{yz}(e_{yz}, e_{yy}) = b_{xz}(e_{xz}, e_{xx})$. Similarly, we can obtain the other equality.
	\end{proof}
	
	\begin{corollary}
		For any $x < y \in P$, we have $b_{xy}(e_{xy}, e_{xx}) = -b_{xy}(e_{xx}, e_{xy})$.
	\end{corollary}
	\begin{proof}
		Let $x < y \in P$. There exists $z \neq x, y \in P$ that is comparable with $x$ and $y$ by the assumption that any maximal chain in $P$ has at least three elements. 
		We consider three cases.
		
		If $x < y < z$, 
		according to Lemma \ref{equal_1}, we have 
		\begin{equation}
			b_{xy}(e_{xy}, e_{yy}) = b_{yz}(e_{yy}, e_{yz}). \label{eq:b=-b_3}
		\end{equation}
		From Lemma \ref{equal_2}, we have
		\begin{equation}
			b_{xy}(e_{xx}, e_{xy}) = b_{yz}(e_{yy}, e_{yz}). \label{eq:b=-b_2}
		\end{equation}
		
		Considering \eqref{eq:b=-b_2} and \eqref{eq:b=-b_3}, we have
		\begin{equation}
			b_{xy}(e_{xy}, e_{xx}) = -b_{xy}(e_{xy}, e_{yy}) = -b_{yz}(e_{yy}, e_{yz}) = -b_{xy}(e_{xx}, e_{xy}).
		\end{equation}
		
		For the remaining two cases $x < z < y$ and $z < x < y$, a similar process yields the same result.
	\end{proof}

	\begin{theorem}\label{equal}
		For any $x < y$, $u < v$ in $P_j$, $j \in \mathcal{J}$ where $\mathcal{J}$ is the index set defined in \eqref{2_L_decomposition}, we have 
		\begin{equation}
			b_{xy}(e_{xy}, e_{xx}) = b_{uv}(e_{uv}, e_{uu}). \label{eq:equal}
		\end{equation} 
	\end{theorem}
	\begin{proof}
		Let $x < y$, $u < v \in P_j$, $j \in \mathcal{J}$. If $P_j$ is a totally ordered set, it is evident from Lemma \ref{equal_2} that $b_{xy}(e_{xy}, e_{xx}) = b_{uv}(e_{uv}, e_{uu})$. If $P_j$ is not a totally ordered set, according to the construction of $P_j$, there exist chains $l_0, l_1, \dots, l_n \in L_j$ defined in \eqref{2_L_decomposition} such that $x, y \in l_0$, $u, v \in l_n$, and $l_i \approx l_{i+1}$ for any $i \in \{0, 1, \dots, n-1\}$. Therefore, there exist $x_i < y_i \in l_i \cap l_{i+1}$. Because of $l_i$ is a totally ordered set, thus using Lemma \ref{equal_2}, we have
		\begin{equation}
			b_{xy}(e_{xy}, e_{xx}) = b_{x_0 y_0}(e_{x_0 y_0}, e_{x_0 x_0}) = \dots = b_{x_n y_n}(e_{x_n y_n}, e_{x_n x_n}) = b_{uv}(e_{uv}, e_{uu}).
		\end{equation}
		This proves \eqref{eq:equal}. 
	\end{proof}
	
	For any $j \in \mathcal{J}$, let $x_j < y_j \in P_j$, and define
	\begin{equation}
		\lambda_j = -b_{x_j y_j}(e_{x_j y_j}, e_{x_j x_j}) = b_{x_j y_j}(e_{x_j x_j}, e_{x_j y_j}). \label{lambda}
	\end{equation}
	A crucial conclusion is that for any pair $u < v$ in $P_j$, where $j \in \mathcal{J}$, the value $b_{uv}(e_{uv}, e_{uu}) = -\lambda_j$, as demonstrated by the aforementioned theorem.

	Now that the requisite preparations have been completed, we may proceed with the proof of the final theorem.

\begin{theorem}\label{final}
	Let $\mathcal{R}$ be a commutative ring with unity, and let $P$ be a locally finite poset such that any maximal chain in $P$ contains at least three elements. 
	The biderivation $b$ of the incidence algebra $I(P,\mathcal{R})$ is the sum of several inner biderivations and extremal biderivations.
\end{theorem}

\begin{proof}
	We proceed by first considering the case where $P$ is connected. The general case will follow by extending this result to each connected component of $P$.
	
	\textbf{Case 1: $P$ is connected.}
	
	By the decomposition \eqref{eq:2_alpha_decomposition}, we can write $P$ as a union of subsets $P_j$:
	\[
	P = \bigcup_{j \in \mathcal{J}} P_j.
	\]
	
	For any $\alpha, \beta \in I(P, \mathcal{R})$, we can decompose them as follows:
	\begin{align}
		\alpha &= \sum_{j \in \mathcal{J}} \alpha'_j + \alpha^D, \quad \text{where} \quad \alpha'_j = \sum_{x < y \in P_j} \alpha_{xy} e_{xy}, \quad \alpha^D = \sum_{z \in P} \alpha_{zz} e_{zz}, \\
		\beta &= \sum_{j' \in \mathcal{J}} \beta'_{j'} + \beta^D, \quad \text{where} \quad \beta'_{j'} = \sum_{u < v \in P_{j'}} \beta_{uv} e_{uv}, \quad \beta^D = \sum_{w \in P} \beta_{ww} e_{ww}.
	\end{align}
	
	Using these decompositions, we expand the biderivation $b(\alpha, \beta)$:
	\begin{align}\label{eq:final}
		b(\alpha, \beta) &= b(\alpha^D, \beta^D) + \sum_{j, j' \in \mathcal{J}} b(\alpha'_j, \beta'_{j'}) + \sum_{j \in \mathcal{J}} b(\alpha'_j, \beta^D) + \sum_{j' \in \mathcal{J}} b(\alpha^D, \beta'_{j'}).
	\end{align}
	
	\paragraph{I. Evaluating $b(\alpha^D, \beta^D)$:}
	
	Consider
	\[
	b(\alpha^D, \beta^D) = \sum_{z, w \in P} b(\alpha_{zz} e_{zz}, \beta_{ww} e_{ww}).
	\]
	
	From Lemma \ref{cant compare} and Theorem \ref{[e_xy,e_uv]=0}, we have
	\begin{equation}\label{eq:final_1}
		\begin{aligned}
			&b(\alpha^D, \beta^D) \\
			=& \sum_{z \leq w \in P} b(\alpha_{zz} e_{zz}, \beta_{ww} e_{ww}) \\
			=& \sum_{\substack{z < w \\ z \text{ min},\, w \text{ max}}}
			(b(\alpha_{zz} e_{zz}, \beta_{ww} e_{ww}) + b(\alpha_{ww} e_{ww}, \beta_{zz} e_{zz})) + \sum_{\substack{\text{$z$ is min}\\ \text{or max}}}b(\alpha_{zz} e_{zz}, \beta_{zz} e_{zz})
		\end{aligned}
	\end{equation}

	By Corollary \ref{corollary1} and Lemma \ref{d_xy(e_xx)+d_xy(e_yy)=0}, the part of each term in the \eqref{eq:final_1} can be expressed as
	\begin{equation}\label{term b(D,D) 1}
		\begin{aligned}
			&\sum_{\substack{z < w \\ z \text{ min},\, w \text{ max}}}
			(b(\alpha_{zz} e_{zz}, \beta_{ww} e_{ww}) + b(\alpha_{ww} e_{ww}, \beta_{zz} e_{zz}))\\
			=&\sum_{\substack{z < w \\ z \text{ min},\, w \text{ max}}}(\alpha_{zz} \beta_{ww} b_{zw}(e_{zz}, e_{ww}) + \alpha_{ww} \beta_{zz} b_{zw}(e_{ww}, e_{zz})) e_{zw}\\
			=&\sum_{\substack{z < w \\ z \text{ min},\, w \text{ max}}}(\alpha_{zz} \beta_{ww} + \alpha_{ww} \beta_{zz}) b_{zw}(e_{zz}, e_{ww}) e_{zw}
		\end{aligned}
	\end{equation}

	By Theorem \ref{b(e_{xx},e_{xx})} and \ref{d_xy(e_xx)+d_xy(e_yy)=0}, the other part of each term in the \eqref{eq:final_1} can be expressed as
	\begin{equation}\label{term b(D,D) 2}
		\begin{aligned}
			&\sum_{\substack{\text{$z$ is min}\\ \text{or max}}}b(\alpha_{zz} e_{zz}, \beta_{zz} e_{zz}) \\
			=& \sum_{\substack{z < w \\z \text{ min}\ w \text{ max}}}(\alpha_{zz}\beta_{zz}b_{zw}(e_{zz},e_{zz}) + \alpha_{ww}\beta_{ww}b_{zw}(e_{ww},w_{ww}))\\
			=&-\sum_{\substack{z < w \\z \text{ min}\ w \text{ max}}}(\alpha_{zz}\beta_{zz} + \alpha_{ww}\beta_{ww})b_{zw}(e_{zz},e_{ww})e_{zw}
		\end{aligned}
	\end{equation}

	Substituting \eqref{term b(D,D) 1} and \eqref{term b(D,D) 2} back, equation \eqref{eq:final_1} becomes
	\begin{equation}\label{eq:1_final}
		\begin{aligned}
			&b(\alpha^D, \beta^D) \\
			=&\sum_{\substack{z < w \in P \\ z \text{ min},\, w \text{ max}}} (\alpha_{zz} \beta_{ww} + \alpha_{ww} \beta_{zz} - \alpha_{zz}\beta_{zz} - \alpha_{ww}\beta_{ww})b_{zw}(e_{zz}, e_{ww})e_{zw}\\
			=&\sum_{\substack{z < w \in P \\ z \text{ min},\, w \text{ max}}}\left[\alpha_{zz}e_{zz} + \alpha_{ww}e_{ww}, [\beta_{zz}e_{zz} + \beta_{ww}e_{ww}, -b_{zw}(e_{zz},e_{ww})e_{zw}]\right]\\
			=&\sum_{\substack{z < w \in P \\ z \text{ min},\, w \text{ max}}}[\hat{\alpha},[\hat{\beta}, T]],
		\end{aligned}
	\end{equation}
	where $\hat{\alpha} = \sum\limits_{\substack{\text{$z$ is min}\\ \text{or max}}}\alpha_{zz}e_{zz}$,  
	$\hat{\beta} = \sum\limits_{\substack{\text{$z$ is min}\\ \text{or max}}}\beta_{zz}e_{zz}$ and 
	$T = -\sum\limits_{\substack{z < w \\z \text{ min},\, w \text{ max}}}b_{zw}(e_{zz},e_{ww})e_{zw}$

	\paragraph{II. Evaluating $\sum_{j \neq j' \in \mathcal{J}} b(\alpha'_j, \beta'_{j'})$:}
	
	We have
	\[
	\sum_{j \neq j' \in \mathcal{J}} b(\alpha'_j, \beta'_{j'}) = \sum_{j \neq j' \in \mathcal{J}} \sum_{x < y \in P_j} \sum_{u < v \in P_{j'}} b(\alpha_{xy} e_{xy}, \beta_{uv} e_{uv}).
	\]
	
	Consider elements $x < y \in P_j$ and $u < v \in P_{j'}$, where each pair in $x,y,u,v$ is comparable.
	A maximal chain, denoted by $l = \{x, y, u, v\}^+$, exists in $P$ that contains $x, y, u, v$. 
	There exist $l' \in L_i$ and $l'' \in L_j$, where $L_i, L_j$ are defined in \eqref{2_L_decomposition}, such that $x,y \in l'$ and $u, v\in l''$.
	It is obvious that $l' \approx l \approx l''$, then $l$ would belong to both $L_j$ and $L_{j'}$.
	Thus, we get $x,y,u,v \in l \in L_j\bigcap L_{j'}$, which contradicts Lemma \ref{P_iP_j}.
	
	Hence, existing a pair elements among $x, y, u, v$ are not comparable, and by Lemma \ref{cant compare}, it follows that
	\[
	b(\alpha_{xy} e_{xy}, \beta_{uv} e_{uv}) = 0.
	\]
	
	Therefore,
	\begin{equation}\label{eq:2_final}
		\sum_{j \neq j' \in \mathcal{J}} b(\alpha'_j, \beta'_{j'}) = 0.
	\end{equation}

	\paragraph{III. Evaluating $\sum_{j \in \mathcal{J}} b(\alpha'_j, \beta'_{j})$:}
	
	We have
	\[
	\sum_{j \in \mathcal{J}} b(\alpha'_j, \beta'_j) = \sum_{j \in \mathcal{J}} \sum_{x < y, u < v \in P_j} b(\alpha_{xy} e_{xy}, \beta_{uv} e_{uv}).
	\]
	
	From Lemmas \ref{cant compare} and \ref{[e_xy,e_uv]=0}, the part  $b(\alpha_{xy} e_{xy}, \beta_{uv} e_{uv})$ in above equation is non-zero only when any pair in $x,y,u,v$ is compared and $[e_{xy}, e_{uv}] \neq 0$, which implies either $x < y = u < v$ or $u < v = x < y$. Thus, the sum simplifies to
	\begin{equation}\label{eq:final_2}
	\begin{aligned}
		\sum_{j \in \mathcal{J}} b(\alpha'_j, \beta'_j) &= \sum_{j \in \mathcal{J}} \sum_{x < y < z \in P_j} \left( b(\alpha_{xy} e_{xy}, \beta_{yz} e_{yz}) + b(\alpha_{yz} e_{yz}, \beta_{xy} e_{xy}) \right) \\
		&= \sum_{j \in \mathcal{J}} \sum_{x < y < z \in P_j} \left( b_{xz}(\alpha_{xy} e_{xy}, \beta_{yz} e_{yz})  +  b_{xz}(\alpha_{yz} e_{yz}, \beta_{xy} e_{xy}) \right) e_{xz}.
	\end{aligned}
	\end{equation}
	
	Applying Lemmas \ref{d_structure_like_linear} and \ref{equal_1}, we obtain
	\begin{equation}
	\begin{aligned}
		&b_{xz}(\alpha_{xy}e_{xy},\beta_{yz}e_{yz})e_{xz}+b_{xz}(\alpha_{yz}e_{yz},\beta_{xy}e_{xy})e_{xz} \\
		=& (\alpha_{xy}\beta_{yz}b_{xy}(e_{xy},e_{yy})+\alpha_{yz}\beta_{xy}b_{yz}(e_{yz},e_{yy}))e_{xz}\\
		=&\lambda_j(\alpha_{xy}\beta_{yz}-\alpha_{yz}\beta_{xy})e_{xz},
	\end{aligned}
	\end{equation}
	where $\lambda_j$ is defined by \eqref{lambda}. Consequently, equation \eqref{eq:final_2} becomes
	\begin{equation}\label{eq:3_final}
		\sum_{j \in \mathcal{J}} b(\alpha'_j, \beta'_j) = \sum_{j \in \mathcal{J}} \lambda_j \sum_{x < y < z \in P_j} (\alpha_{xy} \beta_{yz} - \alpha_{yz} \beta_{xy}) e_{xz}.
	\end{equation}
	
	\paragraph{IV. Evaluating $\sum_{j \in \mathcal{J}} b(\alpha'_j, \beta^D) + \sum_{j \in \mathcal{J}} b(\alpha^D, \beta'_j)$:}
	
	First, consider $b(\alpha'_j, \beta^D)$:
	\[
	b(\alpha'_j, \beta^D) = \sum_{x < z \in P_j} b(\alpha_{xz} e_{xz}, \beta^D).
	\]
	
	From Lemma \ref{[e_xy,e_uv]=0}, $b(\alpha_{xz} e_{xz}, \beta_{ww} e_{ww}) = 0$ unless $w = x$ or $w = z$. Therefore,
	\begin{equation}
			\begin{aligned}
			b(\alpha'_j, \beta^D) &= \sum_{x < z \in P_j} \left( \alpha_{xz} \beta_{xx} b_{xz}(e_{xz}, e_{xx}) + \alpha_{xz} \beta_{zz} b_{xz}(e_{xz}, e_{zz}) \right) e_{xz} \\
			&= \sum_{x < z \in P_j} (\alpha_{xz} \beta_{xx} - \alpha_{xz} \beta_{zz}) b_{xz}(e_{xz}, e_{xx}) e_{xz} \\
			&= \lambda_j \sum_{x < z \in P_j} (\alpha_{xz} \beta_{zz} - \alpha_{xz} \beta_{xx}) e_{xz}.
		\end{aligned}
	\end{equation}

	Similarly, for $b(\alpha^D, \beta'_j)$, we obtain
	\[
	b(\alpha^D, \beta'_j) = \lambda_j \sum_{x < z \in P_j} (\alpha_{xx} \beta_{xz} - \alpha_{zz} \beta_{xz}) e_{xz}.
	\]
	
	Combining these results, we have
	\begin{equation}\label{eq:4_final}
		\sum_{j \in \mathcal{J}} b(\alpha'_j, \beta^D) + \sum_{j \in \mathcal{J}} b(\alpha^D, \beta'_j) = \sum_{j \in \mathcal{J}} \lambda_j \sum_{x < z \in P_j} (\alpha_{xz} \beta_{zz} - \alpha_{xz} \beta_{xx} + \alpha_{xx} \beta_{xz} - \alpha_{zz} \beta_{xz}) e_{xz}.
	\end{equation}
	
	\paragraph{Combining All Components:}
	Substituting equations \eqref{eq:1_final}, \eqref{eq:2_final}, \eqref{eq:3_final}, and \eqref{eq:4_final} into equation \eqref{eq:final}, we obtain
	\begin{align*}
		&b(\alpha, \beta) - b(\alpha^D, \beta^D)\\
		&= \sum_{j \in \mathcal{J}} \lambda_j \left( \sum_{x < y < z \in P_j} (\alpha_{xy} \beta_{yz} - \alpha_{yz} \beta_{xy}) + \sum_{x < z \in P_j} (\alpha_{xz} \beta_{zz} - \alpha_{xz} \beta_{xx} + \alpha_{xx} \beta_{xz} - \alpha_{zz} \beta_{xz}) \right) e_{xz} \\
		&= \sum_{j \in \mathcal{J}} \lambda_j \sum_{x \leq y \leq z \in P_j} (\alpha_{xy} \beta_{yz} - \alpha_{yz} \beta_{xy}) e_{xz} \\
		&= \sum_{j \in \mathcal{J}} \lambda_j [\alpha_j, \beta_j],
	\end{align*}
	where $\alpha_j=\sum_{x\leq y\in P_j}\alpha_{xy}e_{xy}, \beta_j=\sum_{x\leq y\in P_j}\beta_{xy}e_{xy}$.
	
	Consider the aforementioned points, we deduce the result:
	\begin{equation}
		b(\alpha, \beta) = \sum_{j \in \mathcal{J}} \lambda_j [\alpha_j, \beta_j] + [\hat{\alpha},[\hat{\beta}, T]].
	\end{equation}

	\textbf{Case 2: $P$ is not connected.}
	
	When $P$ is disconnected, we utilize the decomposition \eqref{1_decomposition}:
	\[
	P = \bigcup_{i \in \mathcal{I}} P^i,
	\]
	where each $P^i$ is a connected component of $P$. Further, each connected component $P^i$ can be decomposed as
	\[
	P^i = \bigcup_{j \in \mathcal{J}_i} P^i_j
	\]
	using decomposition \eqref{2_decomposition}.
	
	For any $\alpha, \beta \in I(P, \mathcal{R})$, write
	\[
	\alpha = \sum_{i \in \mathcal{I}} \alpha^i \quad \text{and} \quad \beta = \sum_{i \in \mathcal{I}} \beta^i,
	\]
	where $\alpha^i, \beta^i \in I(P^i, \mathcal{R})$, as per decomposition \eqref{1_alpha_decomposition}. 
	
	By equation \eqref{eq:decomposition1}, the biderivation $b$ satisfies
	\[
	b(\alpha, \beta) = \sum_{i \in \mathcal{I}} b(\alpha^i, \beta^i).
	\]
	
	Since each $P^i$ is connected, applying the result from Case 1, we obtain
	\[
	b(\alpha^i, \beta^i) = \sum_{j \in \mathcal{J}_i} \lambda^i_j [\alpha^i_j, \beta^i_j]  + [\hat{\alpha}^i,[\hat{\beta}^i, T^i]].
	\]
	
	Therefore, combining all components, we conclude that
	\begin{equation}
		b(\alpha, \beta) = \sum_{i \in \mathcal{I}} \left( \sum_{j \in \mathcal{J}_i} \lambda^i_j [\alpha^i_j, \beta^i_j] + [\hat{\alpha}^i,[\hat{\beta}^i, T^i]] \right),
	\end{equation}
	as desired. It is evident that $b$ is the sum of several inner biderivations and extremal biderivations.
\end{proof}

It is evident that if there are not $x < y \in P$ that $x$ is the minimal element and $y$ is the maximal element, $\hat{\alpha}^i = \hat{\beta}^i = T^i = 0$. Consequently, the next corollary holds.

\begin{corollary}
	Let $P$ be a poset that has at least three elements, and let $\mathcal{R}$ be a commutative ring with unity. In the incidence algebra of $P$ over $\mathcal{R}$, if the number of elements in any maximal chain in $P$ is infinite, every additive biderivation is the sum of several inner biderivations.
\end{corollary}

\bigskip
\noindent{\bf Acknowledgements
} C. Zhang would like to thank the support of NSFC grant No. 12101111.

\bibliographystyle{elsarticle-num}

\end{document}